\documentclass[11pt]{amsart}

\usepackage{latexsym}
\usepackage{amssymb}
\usepackage{amsmath}

\newtheorem{theorem}{Theorem}[section]

\newtheorem{proposition}[theorem]{Proposition}
\newtheorem{corollary}[theorem]{Corollary}
\newtheorem{definition}[theorem]{Definition}

\newtheorem{remark}[theorem]{Remark}

\newtheorem{problem}[theorem]{Problem}

\newcommand\supp{\mathop{\rm supp}}

\newcommand\id{\mathop{\rm id}}

\newcommand{\cl}[1]{\mathcal{#1}}

\begin{document}

\title{On the spectrum of multiplication operators}

\author{V. S. Shulman}
\address{Department of Mathematics, Vologda State University, Vologda, Russia}
\email{shulman.victor80@gmail.com}

\author{L. Turowska}
\address{Department of Mathematical Sciences,
Chalmers University of Technology and  the University of Gothenburg,
Gothenburg SE-412 96, Sweden}
\email{turowska@chalmers.se}

\maketitle

\dedicatory{{\it To Yuri Stefanovich Samoilenko with love and gratitude for theorems and songs}}
\begin{abstract}
We study relations between spectra of two operators that are connected to each other through some intertwining conditions. As application we obtain new results on the spectra of multiplication operators on $B(\cl H)$ relating it to the spectra of the restriction of the operators to the ideal $\mathcal C_2$ of Hilbert-Schmidt operators. We also solve one of the problems, posed in \cite{magajna}, about the positivity of the spectrum of multiplication operators with positive operator coefficients when the coefficients on one side commute. Using the Wiener-Pitt phenomena we show that the spectrum of  a multiplication operator with normal coefficients satisfying the Haagerup condition might be strictly larger than the spectrum of its restriction to $\mathcal C_2$.
\end{abstract}
\section{Introduction}

In operator theory often one has to deal with the "same" operators acting on different spaces, as for instance, the restrictions of an operator to (not necessarily closed) invariant subspaces supplied with different norms. The study of the structure, in particular spectral properties, of such operators may be difficult in one space and rather simple in another one. So it is natural to identify links between the operators that allow to connect their structural properties and in this way reduce difficult problems to simple ones. The present paper is devoted to establishing such links. The paper was influenced  by the work of B. Magajna \cite{magajna}, where the author studied the spectrum of L\"uders operators, playing a significant role in quantum information theory. L\"uders operators or L\"uders operations are symmetric multiplication operators with positive operator coefficients on $B(\cl H)$ (the precise definition is given below); in \cite{magajna}  B. Magajna  answered in negative the question, raised in \cite{nagy},  whether the spectra of such operators is always non-negative.

The spectral theory of multiplication operators fits very well in the framework of the subject of the paper as such operators are naturally defined not only on the algebra $B(\cl H)$ of all bounded linear operators on a Hilbert space $\cl H$ but also on its symmetrically normed ideals; the most suitable such ideal is the ideal $\cl C_2$ of Hilbert-Schmidt operators, where many questions on multiplication operators have straightforward solutions. For example, the restriction of a L\"uders operator to $\cl C_2$ is a positive operator on the Hilbert space $\cl C_2$ and hence its spectrum is contained in $\mathbb R_+:=[0,+\infty)$.

In this paper we obtain a new information about the spectra of multiplication operators. In particular, we solve one of the problems, posed in \cite{magajna}, about the spectra of multiplication operators with positive coefficients when the coefficients on one side build a commutative family. For this we study first general questions about relations between spectra in case when the operators (or representations of Banach algebras) are  intertwined by map with trivial kernel or dense range.

Through the paper $B(X)$ denotes the algebra of all bounded linear operators on a Banach space $X$. For $T\in B(X)$ we write $\sigma(T)$ and $\rho(T)$ for the spectrum and the resolvent set of $T$ respectively.

\section{ Intertwining}

Recall that a bounded linear operator $S$ on a Banach space $X$ is {\it generalized scalar} if it admits a $C^{\infty}(\mathbb{R}^2)$-calculation, that is there is a unital continuous homomorphism $f\mapsto f(S)$ from $C^{\infty}(\mathbb{R}^2)$ to $B(X)$ satisfying condition $z(S) = S$, where $z(x,y) = x+iy$. In this case there are numbers $k\in \mathbb{N}$ such that $\|f(S)\| \le C\|f\|_{k,\sigma(S)}$ where for each compact $F\subset \mathbb{C}$, we denote $ \|f\|_{k,F} = \sum_{0\le i,j\le k}\sup_{z=x+iy\in F} |\frac{\partial^{i+j}{f}}{\partial{x}^i\partial{y}^j}|$. The minimal such $k$ is the {\it order} of $S$.
Given arbitrary bounded operator $T$ on $X$, the local resolvent set $\rho_T(x)$ of $T$ at the point $x\in X$ is defined as the union of all open subsets $U$ of $\mathbb C$ for which there is an analytic function $f:U\to X$ which satisfies
$$(T-\lambda)f(\lambda)=x  \text{ for all }\lambda\in U.$$
The local spectrum $\sigma_T(x)$ of $T$ at $x\in X$ is defined as $\sigma_T(x):=\mathbb C\setminus\rho_T(x)$.
The resolvent set $\rho(T)$ is always a subset of $\rho_T(x)$ and hence $\sigma_T(x)\subset\sigma(T)$. For all subsets $F\subset \mathbb C$ we define the local spectral subspace $X_T(F)$ of $T$ to be
$$X_T(F)=\{x\in X:\sigma_T(x)\subset F\}.$$

The following result was proved in \cite[chapter 3.5]{laursen-neumann} for more general operators but under the condition that the intertwining operator $\Phi$ is bounded.
\begin{theorem}\label{Int} Let $X$, $Y$ be Banach spaces and
let $\Phi: X\to Y$ be a linear map with dense range. Suppose that
\begin{equation}\label{intert}
\Phi T =  S \Phi,
 \end{equation}
 where $T$ is a linear operator on $X$ and $S$ is a generalized scalar operator on $Y$. Then $\sigma(S)\subset \overline{\sigma(T)}$.
\end{theorem}
\begin{proof}
For a closed subset $F\subset \mathbb{C}$, let $Y_S(F)$ denote the corresponding local spectral subspace.
It is well known (see e.g. \cite{ColFo}) that $Y_S(F)$ is closed. Let $k$ be the order of $S$. By \cite[Theorem 1.5.4]{laursen-neumann}, for each $p\geq k+3$,
\begin{equation}\label{spectral}
\cap_{\lambda\notin F}(S-\lambda 1)^pY =  Y_S(F).
 \end{equation}
 Now let $F = \overline{\sigma(T)}$. Then, for each $\lambda\notin F$, $(T-\lambda 1)^pX = X$. As  (\ref{intert}) clearly  implies  $\Phi (T-\lambda 1)^p =  (S-\lambda 1)^p \Phi$ for all $\lambda\in \mathbb{C}$, we obtain
\begin{eqnarray*}\Phi(X)& = &\Phi( \cap_{\lambda\notin F}(T-\lambda 1)^p(X)) \subset \cap_{\lambda\notin F}\Phi ((T-\lambda 1)^p(X)) \\
& =& \cap_{\lambda\notin F}(S-\lambda 1)^p(\Phi (X))
\subset  \cap_{\lambda\notin F}(S-\lambda 1)^p(Y)  = Y_F(S),
\end{eqnarray*}
whence $Y = \overline{\Phi (X)}\subset Y_S(F)$. It follows by \cite[Proposition~1.2.20]{laursen-neumann} that $\sigma(S) \subset F$.
\end{proof}

\begin{corollary}\label{Int-normal} Let $\mathcal H$ be a Hilbert space and $X$ be a Banach space.

{\rm (i)} Let $\Phi: X\to \mathcal H$ be a linear map with dense range, $T$ be a linear operator on $X$ and  $N$ be a bounded normal  operator  on $\mathcal H$ such that $\Phi T =  N\Phi$. Then $\sigma(N)\subset \overline{\sigma(T)}$.

 {\rm (ii)} If $\Psi: {\mathcal H}\to X$ is a linear injective bounded map, $T$ is a bounded linear operator on $X$ and $N$ is a bounded normal operator on ${\mathcal H}$ such that
$\Psi N =  T\Psi$
  then $\sigma(N)\subset \sigma(T)$.

\end{corollary}
\begin{proof}
The first statement follows directly from Theorem~\ref{Int}. To see the second one
let $\Phi = \Psi^*: X^*\to \mathcal H^*$, then $\overline{\Phi X^*} = \mathcal H^*$. Clearly $\Phi T^* =  N^*\Phi$ and $N^*$ is a normal operator on $\mathcal H^*$ supplied with the natural scalar product. Since $\sigma(N^*) = \sigma(N)$, $\sigma(T^*) = \sigma(T)$ and $\sigma(T)$ is closed, the result follows from Theorem \ref{Int}.
\end{proof}

\begin{problem}\label{bound}
 Can one remove the conditions that $\Psi$ and $T$ are bounded?
\end{problem}

\begin{remark}\label{rem1}
\emph{ (i) The proof of Theorem \ref{Int}  needs only one property of $S$: if, for some closed subset $F\subset \mathbb{C}$, the subspace $\cap_{\lambda\notin F}(S-\lambda 1)H$ is dense in $H$ then $F$ contains $\sigma(S)$. It would be interesting to clear the class of such operators. }

\emph{ (ii) The results of Theorem \ref{Int} and Corollary \ref{Int-normal} was proved by Krein \cite{Krein} for the case when $N$ is selfadjoint. There are extensions of Krein's results to operators with spectra on a Jordan curve and satisfying restrictions on the norms of resolvent (see \cite{Rus} and references therein).}
\end{remark}

\bigskip

Let $A$ be a unital regular Banach $*$-algebra of functions on a compact set $K$. Let $\pi: A\to B(\mathcal H)$ be an (unital) injective $*$-representation of $A$ on a Hilbert space $\mathcal H$, and $\tau: A\to B(X)$ be a (unital) representation on a Banach space $X$.
We say that an operator $\Phi: X\to \mathcal H$ {\it intertwines $\tau$ with $\pi$} if $\Phi\tau(a) = \pi(a)\Phi$ for all $a\in A$.

\begin{corollary}\label{repr}
Suppose that $\pi$ is injective. If there is a bounded linear operator $\Phi: X\to \mathcal H$ with dense image, which intertwines $\pi$ with $\tau$, then $\sigma(\tau(a)) = \sigma(a)$ for each $a\in A$.
\end{corollary}
\begin{proof}
Clearly, $\tau(a)$ is invertible if $a$ is invertible, whence $\sigma(\tau(a)) \subset \sigma(a)$. Similarly  $\sigma(\pi(a)) \subset \sigma(a)$. Conversely suppose that  $\pi(a)$ is invertible. Since the closure $M$ of $\pi(A)$ is a C*-algebra, $\pi(a)$ is invertible in $M$. This means that $\chi(\pi(a))\neq 0$ for all $\chi\in \Omega(M)$, where $\Omega(M)$ is the set of all characters of $M$. Let $\pi^*: \Omega(M)\to K$ be the map adjoint to $\pi$: $\pi^*(\chi)(x) = \chi(\pi(x))$ for all $x\in A$. Then $K_0:= \pi^*(\Omega(M))$ is a compact set of $K$.  If  $K_0\neq K$ then by regularity of $A$ there is $0\neq c\in A$ with $c(t) = 0$ for all $t\in K_0$. Then $\chi(\pi(c)) = 0$ for all $\chi\in \Omega(M)$, whence $\pi(c) = 0$. This contradicts to the injectivity of $\pi$. It follows that $K_0 = K$ and therefore $a(t)\neq 0$ for all $t\in K$. Thus $a$ is invertible. We proved that  $\sigma(\pi(a)) = \sigma(a)$ for all $a\in A$.

 Clearly, the operator $\pi(a)$ is normal whence $\sigma(\pi(a))\subset \sigma(\tau(a))$ by Corollary \ref{Int-normal}.  Therefore
$$\sigma(a)=\sigma(\pi(a))\subset\sigma(\tau(a))\subset\sigma(a),$$
giving the statement.
 \end{proof}

 \begin{problem}\label{non-comm} For which non-commutative Banach $*$-algebras this is true?
 \end{problem}

The answer is positive for C*-algebras. We will deduce it from a much more general result.

\begin{theorem}\label{Cstar} Any representation $\tau$ of a C*-algebra $A$ on a Banach space $X$ is a topological isomorphism of $A/\ker(\tau)$ onto $\tau(A)$.
\end{theorem}
\begin{proof} Since $A/\ker(\tau)$ is a C*-algebra we may assume that $\tau$ is injective.

Let us firstly show that $r(\tau(a)) = r(a)$ for each positive element $a\in A$ (where by $r$ we denote the spectral radius). The inequality $r(\tau(a)) \le r(a)$ follows from the evident inclusion $\sigma(\tau(a)) \subset \sigma(a)$. Note that the operator $T = \tau(a)$ admits calculus of continuous functions ($f(T) := \tau(f(a)$).  If $r(\tau(a)) <  r(a)$, choose a number $t_0$ between $r(\tau(a))$ and $ r(a)$ and a continuous function $f$ on $\mathbb{R}$ with $f(t) = 0$ for $t\in (-\infty,t_0)$, $f(t) = 1$ for $t\ge r(\tau(a))$. Then $\tau(f(a)) = f(\tau(a)) = 0$ while $f(a)\neq 0$. This contradicts to the injectivity of $\tau$.

Now for arbitrary $a\in A$ we have
\begin{eqnarray*}\|a\|^2 &=& \|a^*a\| = r(a^*a) = r(\tau(a^*a)) \le \|\tau(a^*a)\|\\
 & =& \|\tau(a^*)\tau(a)\| \le \|\tau(a^*)\|\|\tau(a)\|\le \|\tau\|\|a\|\|\tau(a)\|,\end{eqnarray*}
whence $\|a\|\le \|\tau\|\|\tau(a)\|.$ We proved that the map $\tau(a)\mapsto a$ is continuous.
\end{proof}
For the case when $X$ is a Hilbert space, this result was obtained by Pitts \cite{pitts}, whose proof was based on the Dixmier-Day Theorem.

\begin{corollary}\label{spectral} Let $\pi$ be a $*$-representation of a C*-algebra on a Hilbert space, then for any Banach space representation $\tau$ with $\ker \tau = \ker \pi$, one has
$$\sigma(\pi(a)) = \sigma(\tau(a)) \text{  for all  } a\in A.$$
\end{corollary}
\begin{proof}
The equality immediately follows from Theorem \ref{Cstar} if $\sigma(\tau(a))$ is the spectrum of $\tau(a)$ in $\tau(A)$. To see that the latter coincides with the spectrum of $\tau(a)$ in $B(X)$ it suffices to show that if $\tau(a)$ is invertible in $B(X)$, then $a$ is invertible in $A/\ker{\pi}$. Replacing   $A$ by $A/\ker{\pi}$, suppose that $a$ is not invertible in $A$. Then either $a^*a$ or $aa^*$ is not invertible. In the first case $|a| = (a^*a)^{1/2}$ is not invertible whence there is a
sequence $x_n \in A$ with $\|x_n\| = 1 $ and $\||a|x_n\|\to 0$  whence $\|ax_n \|\to 0$. It follows that $\|\tau(a)\tau(x_n)\|\to 0$. Since $\inf_n\|\tau(x_n)\| > 0$, the operator $\tau(a)$ is not invertible. Similarly one treats the case when $aa^*$ is not invertible.
\end{proof}

It remains to note that if representations  $\pi$ and $\tau$ are intertwined by a (non necessarily continuous) injective operator $W$ with dense range, then their kernels coincide. If $W$ is not injective then $\sigma(\tau(a)) = \sigma(\pi_1(a))$, where $\pi_1$ is the restriction of $\pi$ to $(\ker W)^{\bot}$.

\medskip

\section{Approximate inverse intertwinings}

Let us say that a Banach space $X$ is supplied with a weak topology $\omega$ if in its adjoint $X^*$ a $*$-weakly dense closed subspace $M$ is chosen and $\omega = \sigma(X,M)$. We denote by $B_\omega(X)$ the space of all  $\omega$-continuous operators  $T\in B(X)$ (clearly $T\in B_{\omega}(X)$ if and only if $T^*M\subset M$). Similarly $B_{\omega}(X,Y)$ is the subspace of $B(X,Y)$ consisting of operators continuous with respect to the weak topologies in $X$ and $Y$.

\begin{definition}\label{norm-int}\rm
Let $X, Y$ be Banach spaces with fixed weak topologies, and let $T\in B_{\omega}(X)$, $S\in B_{\omega}(Y)$ be operators intertwined by some operator $\Psi\in B_{\omega}(X,Y)$ :
$\Psi T= S\Psi$.  A sequence of operators $F_n\in B_{\omega}(Y,X)$ is called an {\it approximate inverse intertwining} (AII) for the triple $(T,S,\Psi)$  with respect to these topologies if $F_n\Psi \to 1_X$, $\Psi F_n \to 1_Y$ and $TF_n - F_nS \to 0$ in the point-weak topologies.
\end{definition}

\begin{theorem}\label{point}
If an intertwining triple $(T,S,\Psi)$ has an approximate inverse intertwining then all eigenvalues of $S$ belong to $\sigma_{B_{\omega}(X)}(T)$.
\end{theorem}
\begin{proof}
Let $Sy = \lambda y$ for some non-zero $y\in Y$ and $\lambda\in \mathbb{C}$.
Then $$(T-\lambda)F_ny = F_n(S-\lambda)y + (TF_n-F_nS)y = (TF_n-F_nS)y \to 0.$$ Assume $\lambda\notin \sigma_{B_{\omega}(X)}(T)$. Then $(T-\lambda)^{-1}$ is $\omega$-continuous, giving  $F_ny\to 0$ and therefore $\Psi F_ny \to 0$. As $\Psi F_ny\to y$, we get $y = 0$. A contradiction.
\end{proof}

Note that if $X$ is reflexive then $B_{\omega}(X) = B(X)$ and $\sigma_{B_{\omega}(X)}(T)= \sigma(T)$ which simplifies the statement of Theorem \ref{point}.

\begin{problem}\label{FullSp}
 Is it true that $\sigma(S)\subset \sigma(T)$, under the  assumptions of Theorem \ref{point}?
\end{problem}
For a Banach space $X$, let $\ell_{\infty}(X)$ be the space of all bounded sequences with entries in $X$ and $c_0(X)$ be the subspace of all sequences convergent (in norm) to $0$. Let $\widetilde{X} = \ell_{\infty}(X)/c_0(X)$. Each $T\in B(X)$ defines, by componentwise action, an operator  on $\ell_\infty(X)$ which leaves $ c_0(X)$ invariant, and hence induces a bounded operator $\tilde T$ on the quotient space $\tilde X$.
\begin{definition}\rm
 We say that an intertwining triple $(T,S,\Psi)$ is {\it strongly approximate invertible} if the corresponding triple $(\widetilde{T},\widetilde{S},\widetilde{\Psi})$ has an approximate inverse intertwining.
\end{definition}

\begin{corollary}\label{StrInt}
If an intertwining triple $(T,S,\Psi)$ is strongly approximate invertible then $\sigma_{\rm ap}(S)\subset \sigma(T)$, where $\sigma_{\rm ap}(S)$ is the approximate spectrum of $S$.
If, in addition, $\Psi$ has dense range then $\sigma(S)\subset\sigma(T)$.
\end{corollary}
\begin{proof}  It is easy to see that if $\lambda\in \sigma_{\rm ap}(S)$ then $\lambda$ is an eigenvalue for $\widetilde{S}$. By Theorem \ref{point}, $\lambda\in \sigma(\widetilde{T})$ and hence $\lambda\in \sigma(T)$.

If $\Psi$ has dense range  then $\Psi^*\in B(Y^*, X^*)$ is injective. As $\Psi^*S^*=T^*\Psi^*$,  any eigenvalue $\lambda$ of $S^*$ is an eigenvalue of $T^*$ and hence in $\sigma(T^*)=\sigma(T)$. Since the residual spectrum $\sigma_r(S)$ of $S$ is a subset of the point spectrum of $S^*$ and $\sigma(S)=\sigma_{\rm ap}(S)\cup \sigma_r(S)$, we obtain $\sigma(S)\subset\sigma(T)$.
\end{proof}


Let ${\mathbb A}=\{A_j\}_{j\in J}$, ${\mathbb B}=\{B_j\}_{j\in J}$ ($J$ is finite or countable) be  families of operators on a Hilbert spaces $\mathcal H$ such that
\begin{equation}\label{coef}
\sum_{j\in J}||A_j||^2<\infty \text{ and }\sum_{j\in J}||B_j||^2<\infty.
\end{equation}
Then one can define a {\it multiplication operator} $\Delta:B(\cl H)\to B(\cl H)$ by
\begin{equation}\label{multop}
\Delta(X)=\sum_{j\in J}A_jXB_j.
\end{equation}
Clearly, $\Delta$ is bounded and preserves all symmetrically normed ideals of $B(\cl H)$. Let $ \Delta_{\cl C_2}$ be  the restriction of $\Delta$ to the ideal of all Hilbert-Schmidt operators $\cl C_2$ on $\cl H$. We are interested in the relations between the spectra of the operators $\Delta$ and $\Delta_{\cl C_2}$.

Let $\Psi_2: \cl C_2\to B(\cl H)$ be the identity inclusion. Then $(\Delta_{\cl C_2}, \Delta, \Psi_2)$ is an intertwining triple.

The following condition is close   in spirit to Voiculescu's notion of quasidiagonality with respect to a symmetrically normed ideal
\cite{voi2}.

\begin{definition}\label{2-semi}\rm
We say that a family ${\mathbb A} = \{A_j\}_{j\in J}$ of operators is $2$-semi\-dia\-gonal if
there exists a sequence of projections $P_n$ of finite rank such that
$P_n\to 1$ in the strong operator topology,  and $$\sup_n\sum_{j\in J}||[A_j,P_n]||_{\cl C_2}^2<\infty.$$
\end{definition}

\begin{theorem} \label{aii}
 Suppose that (\ref{coef}) holds and the family $\mathbb A=\{A_j\}_{j\in J}$ is $2$-semidiagonal.
Then all eigenvalues of $\Delta$ are contained in $\sigma(\Delta_{\cl C_2})$.
\end{theorem}

\begin{proof} It was established in \cite{reine} that under these assumptions there exists an approximate inverse intertwining for the triple $(\Delta_{\cl C_2},\Delta,\Psi_2)$, with respect to the usual weak topology in $\cl C_2$ and $*$-weak topology $\sigma(B(\cl H),{\cl C_1}(\cl H))$. It remains to apply Theorem \ref{point} taking into account that $\cl C_2$ is reflexive.
\end{proof}

The conditions that imply 2-semidiagonality of  $\mathbb A$ were studied in \cite{reine}. We mention only two of them:

1) the matrices of all $A_j$, with respect to some basis in $\cl H$, are supported by a finite number of diagonals;

2) $\mathbb A$ is a family of commuting normal operators  such that $\mathbb A$ has  finite muliplicity and
$\text{ess-dim }{\mathbb A}\leq 2$, the latter is the essential Hausdorff dimension of the joint spectrum of $\mathbb A$, defined in \cite{reine}.
In particular, if all $A_j$ are Lipschitz functions of a Hermitian operator then $\text{ess-dim }{\mathbb A}\leq 2$.

\section{ L\"uders operators}

The {\it formally adjoint} operator $\widetilde{\Delta}$  of $\Delta$ is defined by the formula $$\widetilde{\Delta}(X)=\sum_{j\in J}A_j^*XB_j^*.$$ The restrictions $\Delta_{\cl C_2}$, $\widetilde{\Delta}_{\cl C_2}$ of operators $\Delta$ and $\widetilde{\Delta}$ to the ideal $\cl C_2$ are adjoint to each other as operators on the Hilbert space $\cl C_2$. We say that $\Delta$ is {\it formally selfadjoint} if $\widetilde{\Delta} = \Delta$, and {\it formally normal} if $\widetilde{\Delta}\Delta = \Delta\widetilde{\Delta}$. Clearly, $\Delta$ is formally normal (resp. formally selfadjoint) if and only if $\Delta_{\cl C_2}$ is normal (resp.  selfadjoint).

One has many ways to distinct "positive" operators among formally selfadjoint ones. Let us say that $\Delta$ is $\cl C_2${\it -positive} if $\Delta_{\cl C_2} \ge 0$ as an operator on the Hilbert space $\cl C_2$. Furthermore, $\Delta$ is {\it formally positive} if $\Delta = \widetilde{\Lambda}\Lambda$ for some multiplication operator $\Lambda$. Clearly, any formally positive operator is $\cl C_2$-positive.  Another important subclass of the class of $\cl C_2$-positive operators consists of operators with positive coefficients ($A_j\ge 0$ and $B_j\ge 0$ for all $j\in J$). If $J$ is finite and $A_j = B_j$ for all $j\in J$ then such operators are called {\it L\"uders operators}.

We are interested in relations between the spectra of the operators $\Delta$ and $\Delta_{\cl C_2}$. In particular, under which conditions the spectra of a formally selfadjoint operators is real, and the spectra of a $\cl C_2$-positive operator is non-negative? The fact that this is not always true was established by B. Magajna \cite{magajna}, who constructed examples of L\"uders operators with non-real spectra.

A similar construction can be used to show that formally positive operators can have non-real eigenvalues. In fact, given $\lambda\in\mathbb C$, by \cite[Corollary 2.5]{magajna} there exist positive operators $A_j$, $B_j$, $j=1,2,3$, on a Hilbert space $\cl H$, such that
$\lambda I=\sum_{j=1}^3 A_jB_j$. By letting $\Lambda(X)=\sum_{j=1}^3A_jXB_j$ we obtain $\tilde\Lambda\Lambda(I)=\lambda^2I$.

Note that if the spectrum of a formally selfadjoint operator $\Delta$ is not positive then the same is true for its restriction $\Delta_{{\cl C}_1}$ to the ideal $\cl C_1$ (the trace class), because $\Delta_{{\cl C}_1}$ is the predual of $\Delta$ with respect to the duality $B(\cl H) = \cl C_1^*$. In its turn the restriction of $\Delta$ to the ideal $\cl K(\cl H)$ of all compact operators is the predual of $\Delta_{{\cl C}_1}$, so the spectrum of $\Delta_{{\cl K}(\cl H)}$ also needs not be positive or even real. On the other hand, as $\cl C_1\subset \cl C_2$, all eigenvalues of $\Delta_{{\cl C}_1}$ are positive whenever $\Delta$ is $\cl C_2$ positive.
\begin{problem}\label{restr} Let $\Delta$ be a L\"uders operator.

a) Can $\Delta_{{\cl K}(\cl H)}$ have non-positive eigenvalues?

b) Is it true that $\sigma(\Delta_{{\cl C}_p}) \subset \mathbb{R_+}$, for each $p\in (1,\infty)$?
\end{problem}

In search of conditions that provide the positivity of the spectrum for a L\"uders operator $\Delta: X\mapsto \sum_{j=1}^nA_jXB_j$, B. Magajna considered the case when the left coefficients $A_1,...,A_n$ are commuting. He proved that this condition of "one-sided commutativity" is sufficient if $n=2$, and asked if the same is true for all $n$. Now we will show that the answer to this question is affirmative even for operators of infinite length.

\begin{theorem}\label{tens} Let $W = A\widehat{\otimes}B$, where $A = C(K)$, ($K$ is compact), $B$ is a unital Banach algebra and $\hat\otimes$ denotes the projective tensor product. Then for each
 $w =  \sum_i f_i\otimes b_i\in W$ one has
$$\sigma(w) = \cup_{t\in K}\sigma(\sum_i f_i(t)b_i).$$
\end{theorem}
\begin{proof} Let $\text{Prim}(W)$ be the set of all primitive ideals of $W$. For each $I\in \text{Prim}(W)$, let $q_I$ be the quotient map $W \to W/I$. By theorem of Zemanek \cite{zem}, $$\sigma(w) = \cup_{I\in \text{Prim}(W)}\sigma(q_I(w)).$$
If $\pi_I$ is an irreducible representation of $W$ with kernel $I$ then $\pi_I(A\otimes 1)$ is in the center of $\pi_I(W)$. Since $\pi_I(A\otimes 1)$ is isomorphic to $A/J$ where $J = \{f\in A: f\otimes 1 \in I\}$, $A/J$ is either one-dimensional or there are non-zero operators $T_1,T_2\in \pi_I(A\otimes 1)$ with $T_1T_2 = 0$. Setting $Y = \ker T_1$ we see that $Y$ is a non-trivial subspace invariant for $\pi_I(W)$. So $A/J$ is one-dimensional, $J = \{f\in A: f(t) = 0\}$, for some $t\in K$.

 It follows that $\pi_I(f\otimes 1) = f(t)1$ and $\pi_I(f\otimes b) = f(t)\tau(b)$, for all $f\in C(K)$, $b\in B$ , where $\tau$ is an irreducible representation  of $B$. Therefore $\pi_I(w) = \sum_if_i(t)\tau(b_i) = \tau(\sum_if_i(t)b)$ whence
$$\sigma(q_I(w)) = \sigma(\pi_I(w)) = \sigma(\tau(\sum_if_i(t)b_i)) \subset \sigma(\sum_if_i(t)b_i).$$
 It follows that
$$\sigma(w) \subset \cup_{t\in K}\sigma(\sum_i f_i(t)b_i).$$
The converse inclusion is evident.
\end{proof}

\begin{corollary}\label{luders} Let $\Delta(X) = \sum_{i=1}^{\infty}A_iXB_i$, with $A_i\ge 0$, $B_i\ge 0$ such that $\sum_{i=1}^\infty\|A_i\|^2<\infty$, $\sum_{i=1}^\infty\|B_i\|^2<\infty$. If $A_iA_j=A_jA_i$ for all $i,j$, then $\sigma(\Delta)\subset \mathbb{R}_+$.
\end{corollary}
\begin{proof} Let $A$ be the unital C*-algebra generated by all $A_i$. Then there is an isometric isomorphism $\phi$ of $A$ onto $C(K)$, where $K$ is a compact. Let $W = C(K)\widehat{\otimes}B(\cl H)$ and $\pi = \phi\otimes \id$ be the representation of $W$ on $B(\cl H)$ that sends $f\otimes T$ to the operator $X\mapsto \phi(f)XT$. Then $\Delta = \pi(w)$, where $w =  \sum_{i=1}^\infty f_i\otimes B_i$, $f_i = \phi^{-1}(A_i)$. It follows that $\sigma(\Delta)\subset \sigma(w) = \cup_{t\in K}\sigma(\sum_i f_i(t)B_i)$. Since $f_i(t)\ge 0$ for all $i$, the operators $ f_i(t)B_i$ are positive for all $i$ and $\sigma(\sum_{i=1}^\infty f_i(t)B_i)\subset \mathbb{R}_+$. Hence $\sigma(\Delta)\subset \mathbb{R}_+$.
\end{proof}

The most well studied class of formally normal operators are operators with commutative normal coefficients. If $J$ is finite their spectra are described in even more general situation \cite{curto_fialkow}. The answer for infinite $J$ is the same.

\begin{proposition}\label{ComNor}
Let $\mathbb A$ and $\mathbb B$ be commutative families of normal operators satisfying  (\ref{coef}). Then $$\sigma(\Delta) = \{\mathbb{\lambda}\cdot\mathbb{\mu}: \mathbb{\lambda}\in \sigma(\mathbb{A}), \mathbb{\mu}\in\sigma(\mathbb{B})\}=\sigma(\Delta_{\cl C_2}).$$
\end{proposition}
\begin{proof}
The subsets $X = \sigma(\mathbb{A})$,  $Y = \sigma(\mathbb{B})$ of $\ell^2(J)$ are compact in the weak topology of $\ell^2(J)$, the functions $\tau_j(\mathbb{\lambda}) := \lambda_j$, $\eta_j(\mathbb{\mu}): = \mu_j$, $j\in J$,  are continuous so that the function $$F(\mathbb{\lambda},\mathbb{\mu}): = \mathbb{\lambda}\cdot\mathbb{\mu} = \sum_j \tau_j(\mathbb{\lambda})\eta_j(\mathbb{\mu})$$ belongs to the Varopoulos algebra $V(X,Y)=C(X)\hat\otimes C(Y)$.

Define now representations $\tau: V(X,Y)\to B(B(\mathcal H))$ and $\pi:V(X,Y)\to B(\cl C_2(\mathcal H))$ by letting for $\Phi(x,y)=\sum_{i=1}^\infty f_i(x)g_i(y)$
$$\tau(\Phi)(T)=\sum_{i=1}^\infty f_i(\mathbb A)Tg_i(\mathbb B), \ T\in B(\cl H),$$ and $\pi(\Phi)=\tau(\Phi)|_{\cl C_2(\mathcal H)}$.
Let $\Psi_2: \cl C_2(\mathcal H)\to B(\mathcal H)$ be the inclusion map. Then $\tau(\Phi)\Psi_2=\Psi_2\pi(\Phi)$. By Corollary \ref{repr}, $\sigma(\Delta)=\sigma(F)=\sigma(\Delta_{\cl C_2})$.
\end{proof}

\medskip

  In the next section we will show that the statement of Proposition~\ref{ComNor} is false for more general class of multiplication operators with commuting normal coefficients.

\smallskip

\section{ Mutiplication operators with the Haagerup condition}

It is known  that  a multiplication operator (\ref{multop}) is well  defined and bounded on $B(\cl H)$ if its coefficients satisfy a more general condition:
\begin{equation}\label{Haag}
\|\sum_{j\in J}A_jA_j^*|| < \infty, \|\sum_{j\in J}B_j^*B_j||  <\infty,
\end{equation}
where the convergence of the series and the expression for the multiplication operators is in the weak* topology.  In this generality it is possible that $\Delta$ does not preserve  the Schatten-von Neumann ideals $\cl C_p$, but if $A_j$ and $B_j$ are normal then this is true. Indeed in this case the operator $\widetilde{\Delta}$ is also bounded on $B(\cl H)$ whence by duality $\Delta$ preserve $\cl C_1$ and the invariance of all $\cl C_p$  can be easily seen by using a complex interpolation argument.




\medskip
Let now the coefficient families  ${\mathbb A}$ and ${\mathbb B}$ be commutative and consist of normal operators.
Realizing them as families of multiplication operators on $H_1=L_2(Y,\nu)$ and $H_2=L_2(X,\mu)$ respectively, i.e.
$$A_jg(y)=g_j(y)g(y),\quad B_jf(x)=f_j(x)f(x),$$
we obtain
\begin{equation}\label{coef2}
\text{ess sup}_{x\in X}\sum_{j\in J}|f_j(x)|^2<\infty\quad \text{ess sup}_{y\in Y}\sum_{j\in J} |g_j(y)|^2<\infty.
\end{equation}
Hence $F=\sum_{j\in J} f_j\otimes g_j$ is an element of the weak$^*$ Haagerup tensor product
$$V^\infty(X,Y):=L_\infty(X,\mu)\otimes_{w*h}L_\infty(Y,\nu)$$ (see \cite{blecher-smith} for the definition of this tensor product).
It is known that elements $\varphi= \sum_{i=1}^\infty a_i\otimes b_i\in V^\infty(X,Y)$ can be identified with (marginally equivalence classes of) functions $F:X\times Y\to\mathbb C$,
$$\varphi(x,y)=\sum_{i=1}^\infty a_i(x)b_i(y).$$
Recall that a subset $E\subset X\times Y$ is called marginally null (with respect to $\mu\times \nu$) if
$E\subset (X_1\times Y)\cup (X\times Y_1)$ and $\mu(X_1)=\nu(Y_1)=0$.

Set $\Gamma(X,Y)=L_2(X,\mu)\hat\otimes L_2(Y,\nu)$. We can also identify $\Gamma(X,Y)$ with the space of all (marginally equivalence classes of) functions  $h:X\times Y\to\mathbb C$ which admit representation
$$h(x,y)=\sum_{i=1}^\infty u_i(x)v_i(y),$$
where $u_i\in L^2(X,\mu)$, $v_i\in L_2(Y,\nu)$, such that $\sum_{i=1}^\infty||u_i||_2^2<\infty$,  and $\sum_{i=1}^\infty||v_i||_2^2<\infty$. We write $\|h\|_{\Gamma}$ for the projective norm of $h\in \Gamma(X,Y)$. It is known that $B(L_2(X,\mu), L_2(Y,\nu))$ is dual to $\Gamma(X,Y)$ and the duality is given by
$$\langle X,f\otimes g\rangle=\langle Xf,\bar g\rangle, $$
for $ X\in B(H_1,H_2)$, $f\in L_2(X,\mu)$, $g\in L_2(Y,\nu)$.
One has
$$V^\infty(X,Y)=\{\varphi\in L^\infty(X\times Y): \varphi h\in^{\mu\times\nu}\Gamma(X,Y) \ \forall h\in\Gamma(X,Y)\},$$
here $\varphi h\in^{\mu\times\nu}\Gamma(X,Y)$ means that $\varphi h$ differs from a function in $\Gamma(X,Y)$ on a $\mu\times\nu$-null set.

Let $A(\mathbb R)=\cl FL_1(\mathbb R)$ be the Fourier algebra of $\mathbb R$, where $\cl F$ is the Fourier transform. Then the map
$P:\Gamma(\mathbb R,\mathbb R)\to A(\mathbb R)$
(with the Lebesque measure on $\mathbb R$) given by
\begin{equation}\label{pmap}
P(f\otimes g)(t)=g\ast\check f(t),
\end{equation}
 where $\check f(t)=f(-t)$, is a contractive surjection.

If $B(\mathbb R)$ is  the Fourier-Stiltjes algebra of the group ${\mathbb R}$ then,  for each $h\in  B(\mathbb R)$,  the function $h(x-y)$ belongs to $ V^\infty({\mathbb R},{\mathbb R})$ (with respect to the Lebesgue measure). Recall that $B(\mathbb R)$ is isomorphic to the measure algebra $M(\mathbb R)$ via the Fourier-Stiltjes transform $\hat\mu(t)=\int_{\mathbb R}e^{-itx}d\mu(x)$, $\mu\in M(\mathbb R)$.

The classical fact of harmonic analysis, referred to as the Wiener-Pitt phenomen, is the existence of a function $g\in B({\mathbb R})$ such that $|g(x)|\geq 1$, $x\in\mathbb R$, and $1/g\not\in B(\mathbb R)$ (see e.g. \cite[chapter 4]{laursen-neumann}).

Let $F(x,y)=g(y-x)$.  Then by the above
$$F(x,y)=\sum_{j=1}^\infty f_j(x)g_j(y)$$ with $f_j, g_j\in L^\infty(\mathbb R)$ satisfying (\ref{coef2}).
Let $A_j$, $B_j$ be the multiplication operators by $g_j$ and $f_j$ respectively and $\Delta(X)=\sum_{j=1}^\infty A_jXB_j$.

\begin{proposition}
 $\sigma(\Delta)\ne\{g(x): x\in{\mathbb R}\}=\sigma(\Delta_{C_2})$.
\end{proposition}
\begin{proof}
By the choice of the function $g$ we have that $0\not\in\{g(x): x\in{\mathbb R}\}$.  To prove the claim it is enough to show that $\Delta$ is not invertible. Assume to the contrary that $\Delta$ has an inverse $\Delta'$.
Then, given $\Psi\in \Gamma(\mathbb R,\mathbb R)$, we have
$$\langle\Delta\Delta'(X),\Psi\rangle=\langle\Delta'(X),F\Psi\rangle=\langle X,\Psi\rangle,$$
where $\langle\cdot,\cdot\rangle$ is the pairing between $B(L^2(\mathbb R))$ and $\Gamma(\mathbb R,\mathbb R)$.
As $\text{null } F=\emptyset$, we have by \cite[Corollary 4.3]{sht} that the space $R(S_F):=\{F\Psi:\Psi\in\Gamma(\mathbb R,\mathbb R)\}$ is dense in $\Gamma(\mathbb R,\mathbb R)$. Let $S_F:\Gamma({\mathbb R},\mathbb R)\to \Gamma({\mathbb R},\mathbb R)$, $\Psi\mapsto F\Psi$ and $S_{F^{-1}}$ be the  linear operator defined on $R(S_F)$ by $S_{F^{-1}}(F\Psi)=\Psi$.  From the above equality we have
$$\langle\Delta'(X), F\Psi\rangle=\langle X,S_{F^{-1}}(F\Psi)\rangle,$$
$S_{F^{-1}}$ is bounded and $\Delta'=(S_{F^{-1}})^*$.  Write $S_{F^{-1}}$ also for the extension of
$S_{F^{-1}}$ to $\Gamma(\mathbb R,\mathbb R)$. Then $S_{F^{-1}}$ is
the multiplication by $1/F$. 

 Let $P:\Gamma({\mathbb R},{\mathbb
R})\to A({\mathbb R})$ be the surjective contraction given by  (\ref{pmap}), then $P((1/F)\Psi)=(1/g)P(\Psi)$ for any
$\Psi\in\Gamma(\mathbb R,\mathbb R)$. Therefore, $1/g$ is a
multiplier of the Fourier algebra  $A(\mathbb R)$ and hence $1/g\in B(\mathbb R)$ by
\cite[Theorem 3.8.1]{rudin}. A contradiction.
\end{proof}
Observe that we always have $\sigma(\Delta_{\cl C_2})\subset\sigma(\Delta)$ due to  Corollary \ref{Int-normal} (ii).

Similar arguments can be applied to prove the following
\begin{proposition}
The spectrum of a $\cl C_2$-positive operator $\Delta$ with coefficient families, satisfying (\ref{Haag}) and consisting of commuting normal operators, is not necessarily contained in $\mathbb R_+$.
\end{proposition}
\begin{proof}
We note first that there exists $g\in B(\mathbb R)$ such that $g(x)\geq 1$, $x\in\mathbb R$, but $1/g\notin B(\mathbb R)$. Indeed, assuming contrary to the statement that any such positive $g\in B(\mathbb R)$ is invertible we obtain that whenever $f=f_1+if_2\in B(\mathbb R)$ such that $|f|\geq 1$ and $f_1$ and $f_2$ are real valued,  $f_1$, $f_2\in B(\mathbb R)$, $|f|^2=f_1^2+f_2^2\in B(\mathbb R)$ and $1/f=(f_1-if_2)/|f|^2\in B(\mathbb R)$, a contradiction.

Let now $h=g-1\in B(\mathbb R)$ and set $G(x,y)=g(x-y)$, $H(x,y)=h(x-y)$, $(x,y)\in\mathbb R$. We have
$\sigma(\Delta_H|_{C_2})=\{h(x): x\in \mathbb R\}\subset\mathbb R_+$. However $\Delta_H+1=\Delta_G$ is not invertible by the arguments from the proof of the previous proposition. Hence $-1\in\sigma(\Delta_H)$.
\end{proof}


\end{document}